\documentclass[english,12pt]{amsart}
\usepackage{amsmath,amsfonts,bbm,amsthm}
\usepackage[latin1]{inputenc}  % accents 8 bits dans le source
\usepackage[T1]{fontenc}       % accents dans le DVI
\usepackage[psamsfonts]{amssymb}
\usepackage{latexsym}
\usepackage[dvips]{epsfig}
\usepackage{dcolumn}
\usepackage{tabularx}
\usepackage{longtable}
\usepackage{graphics}
\usepackage{pstricks}
\usepackage{pst-node}
\usepackage{latexsym}
\usepackage{babel}
\usepackage{euscript}

\renewcommand{\theequation}{\thesection.\arabic{equation}}
\newtheorem{theorem}{Theorem}
\newtheorem{lemma}{Lemma}

\newtheorem{remark}{Remark}
\newtheorem{definition}{Definition}

\newcommand{\eqnsection}{
\renewcommand{\theequation}{\thesection.\arabic{equation}}
    \makeatletter
    \csname  @addtoreset\endcsname{equation}{section}
    \makeatother}
\eqnsection
\def\demi{{1\over 2}}

\def\P{{{\Bbb P}}}

\def\R{{{\Bbb R}}}

\def\lll{{{\mathcal L}}}

\def\fff{{{\mathcal F}}}

\def\E{{{\Bbb E}}}

\def\indic{{{\mathbbm 1}}}

\newtheorem{corol}{Corollary}[section]
\newtheorem{rema}{Remark}[section]
\newtheorem{exa}{Example}

\renewcommand{\le}{\leqslant}

\renewcommand{\ge}{\geqslant}

\newcommand{\bal}{\begin{align*}}
\newcommand{\eal}{\end{align*}}
\newcommand{\beq}{\begin{eqnarray*}}
\newcommand{\eeq}{\end{eqnarray*}}
\newcommand{\bte}{\begin{theorem}}
\newcommand{\ete}{\end{theorem}}
\newcommand{\bl}{\begin{lemma}}
\newcommand{\el}{\end{lemma}}
\newcommand{\bd}{\begin{description}}
\newcommand{\ed}{\end{description}}

\newcommand{\bc}{\begin{cases}}
\newcommand{\ec}{\end{cases}}
\newcommand{\bp}{\begin{proof}}
\newcommand{\ep}{\end{proof}}
\newcommand{\bco}{\begin{corol}}
\newcommand{\eco}{\end{corol}}

\newcommand{\sign}{\mathsf{sign}}

\newcommand{\1}{\hbox{1 \hskip -7pt I}}

\newcommand{\iy}{\infty}
\newcommand{\tx}{\text}

\newcommand{\Cst}{\mathsf{Cst}}

\newcommand{\F}{\mathcal{F}}
\newcommand{\Hh}{\mathcal{H}}

\newcommand{\Es}{\mathbb{E}}

\newcommand{\Ec}{\mathcal{E}}

\newcommand{\s}{\sigma}

\renewcommand{\phi}{\varphi}
\newcommand{\tphi}{\tilde{\phi}}
\newcommand{\sphi}{{\sigma\Phi}}
\newcommand{\tpe}{\tilde{\phi}^\varepsilon}
\newcommand{\eps}{\varepsilon}
\def\bdes{\begin{description}}
\def\edes{\end{description}}

\def\tY{{\tilde Y}}

\def\cY{{\check Y}}

\begin{document}
\author[C. Sabot]{Christophe SABOT}
\address{Universit\'e de Lyon, Universit\'e Lyon 1,
Institut Camille Jordan, CNRS UMR 5208, 43, Boulevard du 11 novembre 1918,
69622 Villeurbanne Cedex, France} \email{sabot@math.univ-lyon1.fr}
\author[P. Tarres]{Pierre Tarres}
\address{Ceremade, CNRS UMR 7534 and Universit\'e Paris-Dauphine, Place de Lattre de Tassigny, 75775 Paris Cedex 16, France.} \email{tarres@ceremade.dauphine.fr}
\title[ ]{Inverting Ray-Knight identity}

%\keywords{Ray-Knight theorem, Dynkin's isomorphism, vertex-reinforced jump process} 
\thanks{This work was partly supported by the ANR project MEMEMO2, and the LABEX MILYON}
\thanks{The first author is grateful to DMA, ENS, for his hospitality and financial support while part of this work was done.}

\maketitle
\begin{abstract}
We provide a short proof of the Ray-Knight second generalized Theorem, using a martingale which can be seen (on the positive quadrant) as the Radon-Nikodym derivative of the reversed vertex-reinforced jump process measure with respect to the Markov jump process with the same conductances. Next we show that a variant of this process provides an inversion of that Ray-Knight identity. We give a similar result for the Ray-Knight  first generalized Theorem.
\end{abstract}
\section{Introduction}
\label{intro}
Let  $G=(V,E,\sim)$ be a nonoriented connected finite graph without loops, with conductances $(W_e)_{e\in E}$; define, for all $x$, $y$ $\in V$, $W_{x,y}=W_{\{x,y\}}\indic_{x\sim y}$. 

Let $\Ec$ and $L$ be respectively the associated Markov generator and Dirichlet form defined by, for all $f\in\R^V$,  
\begin{align*}
Lf(x)=&\sum_{y\in V}W_{x,y}(f(y)-f(x))\\
\Ec(f,f)&=\frac{1}{2}\sum_{x,y\in V}W_{x,y}(f(x)-f(y))^2
\end{align*}

Let $x_0\in V$ be a special point that will be fixed throughout the text. let $U=V\setminus\{x_0\}$, and let $P^{G,U}$ be the unique probability on $\R^V$ under which $(\varphi_x)_{x\in V}$ is the centered Gaussian field with covariance $E^{G,U}[\varphi_x\varphi_y]=g_U(x,y)$, where $g_U$ is the Green function killed outside $U$, in other words
$$P^{G,U}=\frac{1}{(2\pi)^{\frac{|U|}{2}}\sqrt{\det G_U}}\exp\left\{-\frac{1}{2}\Ec(\varphi,\varphi)\right\}\delta_0(\phi_{x_0})\prod_{x\in U}d\varphi_x,$$
where $G_U:=g_U(.,.)$, and $\delta_0$ is the Dirac at 0 so that the integration is on $(\phi_x)_{x\in U}$ with $\phi_{x_0}=0$.

Let $\P_{z_0}$ be the law under which $(X_t)_{t\ge0}$ is a Markov Jump Process with conductances $(W_e)_{e\in E}$ (i.e. jump rates $W_{ij}$ from $i$ to $j$ $\in V$) starting at $z_0$ at time $0$, with right-continuous paths and local times at $x\in V$, $t\ge0$, 
\begin{equation}
\label{deflt}
\ell_x(t)=\int_0^t \indic_{\{X_u=x\}}\,du.
\end{equation}
Let $\tau_.$ be the right-continuous inverse of $t\rightarrow \ell_{x_0}(t)$
$$\tau_u=\inf\{t\ge0;\,\ell_{x_0}(t)>u\},\, \tx{ for }u\ge0.$$
Our first aim is to provide a short proof of the generalized second Ray-Knight theorem.

\begin{theorem}[Generalized second Ray-Knight theorem, \cite{ekmrs}]
\label{rk2}
For any $u>0$, 
\begin{align*}
&\left(\ell_x(\tau_u)+\frac{1}{2}\phi_x^2\right)_{x\in V} \tx{ under }\P_{x_0}\otimes P^{G,U}, \tx{ has the same law as}\\
&\left(\frac{1}{2}(\phi_x+\sqrt{2u})^2\right)_{x\in V}\tx{ under }P^{G,U}.
\end{align*}
\end{theorem}
This theorem is due to Eisenbaum, Kaspi, Marcus, Rosen and Shi \cite{ekmrs} and is closely related to Dynkin's isomorphism; see \cite{eisenbaum} for a first relation between Ray-Knight theorem and Dynkin's isomorphism, and \cite{marcus-rosen} for an overview of the subject; see \cite{lejan2,lupu} for related work on the link between Markov loops and the Gaussian Free Field. Note that this result plays a crucial r\^ole in a recent work of Ding, Lee and Peres \cite{dlp} on cover times of discrete Markov processes, 
and in the study of random interlacements, see for instance \cite{sznitman1,sznitman2,sznitman}. 

In Section 2, we give our short proof of Theorem \ref{rk2}, independent of any reference to the VRJP. Similar results would hold for Dynkin and Eisenbaum isomorphism theorems (see Section 5, and \cite{marcus-rosen}, \cite{sznitman}). We note that there is also a non-symmetric version of Dynkin's isomorphism \cite{lejan3} (see also \cite{eisenbaum-kaspi}), which our technique cannot provide as it is.

In Section 3, we explain how the martingale appearing in this proof is 
related to the Vertex reinforced jump process (VRJP). However, note that the proof does not need any reference to the VRJP.

This short proof in fact yields an identity that corresponds to an inversion of the Ray-Knight identity, proved in Section 4. Indeed, Theorem \ref{rk2} gives an identity in law, but
fails to give any information on the law of $(\ell_x(\tau_u), \phi_x)_{x\in V}$ conditioned on $\left(\ell_x(\tau_u)+\frac{1}{2}\phi_x^2\right)_{x\in V}$.
We provide below a process that describes this conditional law. 

Finally, Section 5 yields the equivalent inversion for the generalized first Ray-Knight theorem.

Let $(\Phi_x)_{x\in V}$ be positive reals. As before, we fix the special point $x_0\in V$. 
We consider the continuous time process $(\cY_s)_{s\ge 0}$
with state space $V$ defined as follows.
We set
$$
\check L_i(s)= \Phi_i-\int_{0}^s \indic_{\check Y_u=i} du.
$$
At time $s$, we consider the Ising model $(\s_x)_{x\in V}$ on $G$ with interaction
$$
J_{i,j}(s)=
W_{i,j}\check L_i(s) \check L_j(s).
$$
and with boundary condition $\sigma_{x_0}=+1$. We denote by 
$$F(s)=\sum_{\sigma\in \{-1,+1\}^{V\setminus\{x_0\}}}  e^{\sum_{\{i,j\}\in E} J_{i,j}(s) \sigma_i \sigma_j}
$$
its partition function and by $<\cdot>_s$ its associated expectation, so that for example
$$
<\sigma_x>_s = {1\over F(s)}\sum_{\sigma\in \{-1,+1\}^{V\setminus\{x_0\}}}  \sigma_x e^{\sum_{\{i,j\}\in E} J_{i,j}(s) \sigma_i \sigma_j}>0
$$
The process $\check Y$ is then defined as the jump process which, 
conditioned on the past at time $s$, if $Y_s=i$, jumps from $i$ to $j$ at a rate
$$
W_{i,j} L_j(s) {<\sigma_j>_s\over <\sigma_i>_s}.
$$
and stopped at the time
\begin{eqnarray}\label{S}
S=\sup\{s\ge 0, \; \check L_i(s)> 0 \hbox{ for all $i$}\}.
\end{eqnarray}
Note, using the positivity of the Ising model (see for instance \cite{werner}, proposition 7.1),  that $<\sigma_x>_s>0$ for all $s<S$. Hence the process  $\check Y$ is defined up to time $S$. 

We denote by $\P_{\Phi,z}^{\check Y}$ the law of $\cY$ starting from $z$ and initial condition $\Phi$, and stopped at time $S$.  (Note that this law
also depends on the choice of the "special point" $x_0$).

\begin{lemma}\label{fin}
Starting from any point $z\in V$, the process $\cY$ ends at $x_0$, i.e. 
$$S<\iy \;\;\hbox{  and }\;\; 
\cY_S=x_0, \;\;\; \hbox{$\P_{\Phi,z}^{\check Y}$ a.s.}
$$
\end{lemma} 
This process provides an inversion to the second
Ray-Knight identity, as stated in the following theorem.

\begin{theorem}\label{inverse1}
Let $\ell$, $\phi$, $\tau_u$ be as in Theorem \ref{rk2} and set
$$
\Phi_x = \sqrt{ \phi_x^2+ 2 \ell_x(\tau_u)}.
$$
Under $\P_{x_0}\otimes P^{G,U}$, we have
$$
\lll\left( \phi | \Phi\right)\stackrel{law}{=}  (\s\check L(S)),
$$
where $\check L(S)$ is distributed under $\P_{\Phi,x_0}^{\check Y}$ and,  conditionally on $\check L(S)$, 
 $\sigma$ is distributed according to the distribution of the Ising model with interaction $J_{i,j}(S)=W_{i,j}\check L_i(S)\check L_j(S)$ with boundary condition
 $\sigma_{x_0}=+1$.

\end{theorem}
\begin{remark}
Once $\phi$ is known, then obviously $\ell(\tau_u)=(\Phi^2-\phi^2)/2$ is also known: in other words, Theorem \ref{inverse1} is equivalent to the more precise identity
$$
\lll\left( (\ell(\tau_u), \phi) | \Phi\right)\stackrel{law}{=}  \left( \demi(\Phi^2-\check L^2(S)), \sigma \check L(S)\right),
$$ 
where $\check L(S)$ and $\sigma$ are distributed as in the statement of Theorem \ref{inverse1}.
\end{remark}
The proof of Theorem \ref{inverse1} is given in Section \ref{proof-inverse}. 
Theorem \ref{inverse1} is a consequence of a more precise statement, cf Theorem \ref{inverse2}, which gives the law of $(X_s)_{s\le \tau_u}$
conditionally on $\Phi$.
\section{A New proof of Theorem \ref{rk2}}
\label{sec:pf}
Let $G$ be a positive measurable test function. Letting $d\phi:=\delta(\phi_{x_0})\prod_{x\in U}d\phi_x$ and $C=(2\pi)^{-\frac{|U|}{2}}(\det G_U)^{-1/2}$
be the normalizing constant of the Gaussian free field, we get 
\begin{align}
\label{eq1}
&\Es_{x_0}\otimes P^{G,U}\left[G\left(\ell_x(\tau_u)+\frac{1}{2}\phi_x^2\right)_{x\in V}\right]\\
\nonumber
&=C\,\,\Es_{x_0}\left[\int_{\R^U} G\left(\ell_x(\tau_u)+\frac{1}{2}\phi_x^2\right) \exp\left(-\demi \Ec(\phi,\phi)\right)d\phi\right]
\\
\nonumber
&=C\,\,\Es_{x_0}\left[\sum_{\sigma\in \{+1,-1\}^V, \sigma_{x_0}=+1} \int_{\R_+^U} G\left(\ell_x(\tau_u)+\frac{1}{2}\phi_x^2\right) \exp\left(-\demi \Ec(\sigma\phi,\sigma\phi)\right)d\phi\right]
\end{align}
where in the last equality we decompose the integral according to the possible signs of $\phi$ so that the sum is on $\sigma\in \{+1, -1\}^U$ with $\sigma_{x_0}=+1$. 
In the following we simply write $\sum_\sigma$ for this sum.

The strategy is now to make the change of variables $\Phi=\sqrt{2\ell(\tau_u)+\phi^2}$.
Given $\ell=(\ell_i(t))_{i\in V, t\in\R_+}$, 
%and using the convention that $\sign(0)=1$, 
let
$$
D_{u}:=\{\Phi\in\R_+^V: \,\,\Phi_{x_0}=\sqrt{2u}, \,\, \Phi_x^2/2\ge \ell_x(\tau_u)\tx{ for all } x\in V\setminus\{x_0\}\}
$$
We first make the change of variables 
\begin{align*}
T_{u}: \R_+^V\cap\{\phi_{x_0}=0\}\rightarrow 
&\,\,D_{u}\\
\phi\mapsto&\,\,\Phi=T_u(\phi)=\left(\sqrt{2\ell_i(\tau_u)+\phi_i^2}\right)_{i\in V}.
\end{align*}
which can be inverted by $\phi= \sqrt{\Phi_i^2-2\ell(\tau_u)}$. This yields, letting $d\Phi:=\delta_{\sqrt{2u}}(\Phi_{x_0})\prod_{x\in U}d\Phi_x$, 
\begin{align}
\nonumber &\Es_{x_0}\otimes P^{G,U}\left[G\left(\ell_x(\tau_u)+\frac{1}{2}\phi_x^2\right)_{x\in V}\right]\\
\nonumber &=C\,\,\Es_{x_0}\left[\int_{\R_+^U}\sum_{\sigma} G\left(\left(\frac{\Phi_x^2}{2}\right)\right) \exp\left(-\demi\Ec(\sigma\phi,\sigma\phi)\right)|\mathsf{Jac}(T_{u}^{-1})(\Phi)|\indic_{\Phi\in {D}_u}d\Phi\right]\\
\label{eq3} &=C\,\,\Es_{x_0}\left[\int_{\R_+^U}\sum_{\sigma} G\left(\left(\frac{\Phi_x^2}{2}\right)\right) \exp\left(-\demi\Ec(\sigma\phi,\sigma\phi)\right)
\left(\prod_{x\in U}\frac{\Phi_x}{\phi_x}\right)
\indic_{\Phi\in {D}_u}d\Phi\right].
\end{align}
Note that the Jacobian is taken over all coordinates but $x_0$ ($\Phi_{x_0}=\sqrt{2u}$); if  $\Phi\in{D}_u$, then 
$$|\mathsf{Jac}(T_{u}^{-1})(\Phi))|
=|\mathsf{Jac}(\Phi\mapsto\phi)(\Phi))|=\prod_{x\in U}\frac{\Phi_x}{\phi_x}.$$
Given $\Phi\in \R_+^V$ such that $\Phi_{x_0}=\sqrt{2u}$, we define for all $i\in V$, 
\begin{align}
\label{defT}
&T=\inf\left\{t\ge0;\,\ell_{i}(t)=\frac{1}{2}\Phi_i^2\tx{ for some }i\in V\right\}
\\
\nonumber &\Phi_i(t)= \sqrt{\Phi_i^2-2\ell_i(t)}, \;\; t\le T
\end{align}
so that in (\ref{eq3}) we have $\phi=\Phi(\tau_u)$.
%Now $T_{\ell,\s_u}$ maps $\R^{U}\times\{0\}=\R^V\cap\{\phi_{x_0}=0\}$ into $\tilde{D}_u:=D_{\ell,\s_u}\cap\{\Phi_{x_0}=\sqrt{2u}\}$, and we are interested in the law of $(|%T_{\ell,\s_u}(\phi)_i|)_{i\in V}$. 
An important remark is that
\begin{align}\label{Tsigma} \Phi\in {D}_u\iff X_T=x_0\,\,\,\, \iff T=\tau_u\,\,.\end{align}
Finally, we define for a configuration of signs $\sigma\in \{+1,-1\}^V$ with $\sigma_{x_0}=+1$, 
\begin{align}\label{defmart}
M^{\sphi}_t=\exp\left\{-\frac{1}{2}\Ec(\sphi(t),\sphi(t))\right\}\frac{\prod_{j\ne x_0}\sigma_j\Phi_j(0)}{\prod_{j\ne X_{t}}\sigma_j\Phi_j(t)}.
\end{align}
From \eqref{eq3}--\eqref{Tsigma}, we deduce
\begin{align}
\label{eq4}
\Es_{x_0}\otimes P^{G,U}\left[G\left(\ell(\tau_u)+\frac{1}{2}\phi^2\right)\right]
=C\,\,\int_{\R_+^U}G\left(\frac{1}{2}\Phi^2\right)
\Es_{x_0}\left[\sum_{\sigma} M^{\sphi}_T\indic_{\{X_T=x_0\}}\right]d\Phi.
\end{align}
\bl
\label{l1}
For any $\Phi\in\R^V$ and $\sigma\in \{-1,1\}^V$, the process $(M^\sphi_{t\wedge T})_{t\ge0}$ is a uniformly integrable martingale.
\el
\bp
Consider the Markov process $(\ell (t), X(t))$, which obviously has generator
$
\tilde L(g)(\ell ,x)= (\frac{\partial}{\partial \ell_x} +L)g(\ell,x).
$
Let $f$ be the function defined by 
$
f(\ell,x)=\left(\prod_{y\neq x} \sigma_y \sqrt{\Phi_y^2-2\ell_y}\right)^{-1}
$.
Note that, if $t<T$, 
\begin{align*}
\frac{d}{d t}\Ec(\sphi(t),\sphi(t))=\frac{2}{(\sphi)_{X_t}(t)}L(\sphi(t))(X_t)=2{Lf\over f}(\ell(t),X_t)=2{\tilde Lf\over f}(\ell(t),X_t),
\end{align*}
since $f(\ell,x)$ does not depend on $\ell_x$.
Therefore, for $t<T$,
$$
{M^\sphi_t\over M^\sphi_0}= \frac{f(\ell(t),X(t))}{f(0,x_0)}e^{-\int_0^t {\tilde Lf\over f}(\ell(s),X(s)) ds}
$$
which implies that $M^\sphi_{t\wedge T}$ is a martingale, using for instance lemma 3.2 in \cite{ethier-kurtz} p.174-175.

The condition that $f$ is bounded in that result is not satisfied, but the proof remains true, noting that the following integrability conditions hold (see Problem 22 of Chapter 2, p.92 in \cite{ethier-kurtz}): first, using that $(\tilde Lf/f)(\ell(t),X(t))\ge-\Cst(W)$ and $(|Lf|/f)(\ell(t),X(t))\le\Cst(W)/\Phi_{X_t}(t)$, we deduce
\begin{align*}
\int_0^t {\tilde Lf\over f}(\ell(s),X_s)\exp\left(-\int_0^s{\tilde Lf\over f}(\ell(u),X_u)\,du\right)\,ds
&\le\Cst(W,\Phi)\sum_{i\in V}\int_{\Phi_i(t)}^{\Phi_i}\frac{ds}{\sqrt{s}}\\
&\le\Cst(\Phi,W,|V|).
\end{align*}

Second, $f(\ell(t),X(t))$ can be upper bounded by an integrable random variable, uniformly in $t$.  Indeed, let us consider the extension of process $(X_t)_{t\in\R}$ to $\R$, and take the convention that the local times at all sites are $0$ at time $0$. 

For all $j\in V$, let $s_j$ be the (possibly negative) local time at $j$, at the last jump from that site before reaching a local time $\Phi_j^2/2$ at that site. Let, for all $j\in V$, 
$$m_j=\Phi_j^2/2-s_j.$$

Then the random variables $m_j$, $j\in V$, are independent exponential distributions with parameters $W_j=\sum_{k\sim j}W_{jk}$, since the sequence of local times of jumps from $j$ is a Poisson Point Process with intensity $W_j$. 

Now, for all $t\in \R_+$, $j\ne X_t$, $\Phi_j(t\wedge T)\ge\sqrt{2m_j}$, so that $f(\ell(t),X(t))\le\Cst(\Phi)\prod_{j\in V}m_j^{-1/2}$. 

This enables us to conclude, since $\prod_{j\in V}m_j^{-1/2}$ is integrable, which also implies uniform integrability of $M_{t\wedge T}$, using that 
$|M_{t\wedge T}|\le\Cst(W,\Phi,t) f(\ell(t),X(t)).$
\ep

Let us now consider the process
\begin{eqnarray}\label{defN}
N_t^{\Phi} &=&\sum_{\sigma\in \{-1,+1\}^{V\setminus\{x_0\}}} M_t^{\sigma \Phi}
\\
\nonumber
&=&
\sum_{\sigma\in \{-1,+1\}^{V\setminus\{x_0\}}} \exp\left\{-\frac{1}{2}\Ec(\sigma\Phi(t),\sigma \Phi(t))\right\}\frac{\prod_{j\ne x_0}\sigma_j \Phi_j(0)}{\prod_{j\ne X_{t}}\sigma_j\Phi_j(t)}.
\end{eqnarray}

\bl \label{sign-flip}
For all $x\neq x_0$, we have
\begin{eqnarray} \label{sflip}
N_T^{\Phi}\indic_{\{X_T=x \}}=0
\end{eqnarray}
\el
\bp
Let $\sigma^x$ be the spin-flip of $\sigma$ at $x$ : $\sigma^x=\epsilon^x \sigma$ with $\epsilon^x_y=-1$ if $y=x$ and $1$ if $y\neq x$.  
If $x\neq x_0$ and $X_T=x$, then
$$
M^{\sigma^x\Phi}_T=-M^\sphi_T.
$$
Indeed, since $\Phi_x(T)= 0$, then $\sigma^x\Phi(T)= \sphi(T)$, and the minus sign comes from the numerator of the product term in (\ref{defmart}).
By symmetry, the left-hand side of (\ref{sflip}) is equal to
$$
 \demi \left(\sum_{\sigma\in \{-1,+1\}^{V\setminus\{x_0\}}}\left( M^\sphi_T+M^{\sigma^x\Phi}_T\right)\right) \indic_{\{X_T=x \}}=0.
$$
\ep
It follows from Lemmas \ref{l1} and \ref{sign-flip}, by the optional stopping theorem (using the uniform integrability of $M_{t\wedge T}$), that 
\begin{align*}
&\Es_{x_0}\otimes P^{G,U}\left[G\left(\ell_x(\tau_u)+\frac{1}{2}\phi_x^2\right)_{x\in V}\right]
\\
&=
C\,\,\int_{\R_+^U} G\left(\left(\frac{1}{2}\Phi_x^2\right)_{x\in V}\right)
\Es_{x_0}\left[N^\phi_T\indic_{\{X_T=x_0\}}\right]d\Phi \\
&=C\,\,\int_{{\R_+^U}} G\left(\frac{1}{2}\Phi^2\right)
 \Es_{x_0}[N^\phi_T]d\Phi=C\,\,\int_{\R_+^{U}} G\left(\frac{1}{2}\Phi^2\right) N^\phi_0\,d\Phi
 \\
 &=C\,\,\int_{\R_+^{U}} G\left(\frac{1}{2}\Phi^2\right) \left(\sum_{\sigma} \exp\left\{-\frac{1}{2}\Ec(\sphi,\sphi)\right\}\right)d\Phi\\
 &=C\,\,\int_{\R^{U}} G\left(\frac{1}{2}\Phi^2\right) \left( \exp\left\{-\frac{1}{2}\Ec(\Phi,\Phi)\right\}\right)d\Phi,
\end{align*}
which concludes the proof of Theorem \ref{rk2}.
\section{Link with vertex-reinforced jump process}
The aim of this section is to point out a link between the Ray-Knight identity and a reversed version of the Vertex-Reinforced Jump Process (VRJP). 

It is organized as follows. In Subsection \ref{sec:dvrjp} we compute the Radon-Nykodim derivative of the VRJP,  which is similar to the martingale $M$ in Section \ref{sec:pf}: the computation can be used in particular to provide a direct proof of exchangeability of VRJP. In Subsection \ref{sec:drvrjp} we introduce a time-reversed version of the VRJP, i.e. where the process subtracts rather than adds local time at the site where it stays (see Definition \ref{def:trrjp}).  Then we show in Theorem \ref{trrnd} that its Radon-Nykodim derivative is the martingale $M^\sigma$ in Section  \ref{sec:pf} with positive spins $\sigma\equiv +1$. Note that the ``magnetized'' reversed VRJP, defined in Section 1 and related to the inversion of Ray-Knight in Theorem \ref{inverse1}, involves instead the sum $N^\Phi$, cf (\ref{defN}), of all the martingales $M^\sphi$.

\subsection{The vertex-reinforced jump process and its Radon-Nykodim derivative}
\label{sec:dvrjp}
\begin{definition}
Given positive conductances on the edges of the graph $(W_e)_{e\in E}$ and initial positive local times $(\phi_i)_{i\in V}$, the vertex-reinforced jump process (VRJP) is a continuous-time process $(Y_t)_{t\ge0}$ on $V$, starting at time $0$ at some vertex $z\in V$ and such that, if $Y$ is at a vertex $i\in V$ at time $t$, then, conditionally on $(Y_s, s\le t)$, the process jumps to a neighbour $j$ of $i$ at rate $W_{i,j}L_j(t)$, where
$$L_j(t):=\phi_j+\int_0^t {\textnormal{\1}}_{\{Y_s=j\}}\,ds.$$
\end{definition}
The Vertex-Reinforced Jump Process was initially proposed by Werner in 2000, first studied by Davis and Volkov \cite{dv1,dv2},  then Collevechio \cite{collevecchio2,collevecchio3}, Basdevant and Singh \cite{bs}, and Sabot and Tarr\`es \cite{sabot-tarres}. 

Let $D$ be the increasing functional 
$$D(s)=\frac{1}{2}\sum_{i\in V} (L_i^2(s)-\phi_i^2),$$
define the time-changed VRJP
$$Z_t= Y_{ D^{-1}(t)}.$$
and let, for all $i\in V$ and $t\ge0$,  $\ell^Z_i(t)$ be the local time of $Z$ at time $t$.
\begin{lemma}
\label{markovtime}
The inverse functional $D^{-1}$ is given by
$$
D^{-1}(t)=\sum_{i\in V}(\sqrt{\phi_i^2+2\ell^Z_i(t)}-\phi_i).
$$
Conditionally on the past at time $t$, the process $Z$ jumps from $Z_t=i$ to a neighbour $j$ at rate
$$
W_{i,j}\sqrt{\frac{\phi_j^2+2\ell^Z_j(t)}{\phi_i^2+2\ell^Z_i(t)}}.
$$
\end{lemma}
\begin{proof}
The proof is elementary and already  in \cite{sabot-tarres} [Section 4.3, Proof of Theorem 2 ii)] in a slightly modified version, but we include it here for completeness. 
First note that, for all $i\in V$, 
\begin{equation}
\label{ellx}
\ell^Z_i(D(s))=(L_i^2(s)-\phi_i^2)/2,
\end{equation}
since
$$
(\ell^Z_i(D(s)))'= D'(s)\indic_{\{Z_{D(s)}=i\}}={L_{Y_s}(s)} \indic_{\{Y_s=i\}}.
$$
Hence
$$(D^{-1})'(t)={1\over D'(D^{-1}(t))}
={1\over L_{Z_t} (D^{-1}(t))}
=\frac{1}{\sqrt{\phi_{Z_t}^2+2\ell^Z_{Z_t}(t)}},
$$
which yields the expression for $D^{-1}$.
It remains to prove the last assertion:
\begin{eqnarray*}
\P(Z_{t+dt}=j | \fff_t) &=&
\P(Y_{D^{-1}(t+dt)}=j | \fff_t)\\
&=&
W_{Z_t,j}(D^{-1})'(t) L_{j}(D^{-1}(t)) dt
=
W_{Z_t,j}\sqrt{\frac{\phi_j^2+2\ell^Z_j(t)}{\phi_{Z_t}^2+2\ell^Z_{Z_t}(t)}} dt.
\end{eqnarray*}
\end{proof}
Let $\P_{x_0,t}$ (resp. $\P_{\phi, x_0,t}^{Z}$) be the distribution, starting from $x_0$ and on the time interval $[0,t]$, of the Markov Jump Process with conductances $(W_e)_{e\in E}$ (resp. the time-changed VRJP $(Z_t)_{s\in[0,t]}$ with conductances $(W_e)_{e\in E}$ and initial positive local times $(\phi_i)_{i\in V}$). 
\begin{theorem}
\label{rnd}
The law of the time-changed VRJP $Z$ on the interval $[0,t]$  is absolutely continuous with respect to the law of the MJP $X$
with rates $W_{i,j}$, with Radon-Nikodym derivative given by
$$
{d \P_{\phi,x_0,t}^{Z}\over d\P_{x_0,t}}=
{e^{\demi \left(\Ec(\sqrt{\phi^2+2\ell(t)}, \sqrt{\phi^2+2\ell(t)})
-\Ec(\phi,\phi)\right)}}{\prod_{j\neq x_0}\phi_{j}\over \prod_{j\neq X_t}\sqrt{\phi_j^2+2\ell_j(t)}},
$$
where $\ell_j(t)$ is the local time of $X$ at time $t$ and site $j$ defined in \eqref{deflt}.
\end{theorem}
\begin{proof}
In the proof, we write $\ell$ for the local time of both $Z$ and $X$, since we consider $Z$ and $X$ on the canonical space with
different probabilities. Let, for all $\psi\in\R^V$, $i\in V$, $t\ge0$,
$$F(\psi)=\sum_{\{i,j\}\in E}W_{ij}\psi_i\psi_j,\,\,G_i(t)=\prod_{j\ne i}(\phi_j^2+2\ell_j(t))^{-1/2}.$$
First note that the probability, for the time-changed VRJP $Z$, of holding at a site $v\in V$ on a time interval $[t_1,t_2]$ is 
 $$\exp\left(-\int_{t_1}^{t_2}\sum_{j\sim Z_t}W_{Z_t,j}\frac{\sqrt{\phi_j^2+2\ell_j(t)}}{\sqrt{\phi_{Z_t}^2+2\ell_{Z_t}(t)}} dt\right)
 =\exp\left(-\int_{t_1}^{t_2}d\left(F(\sqrt{\phi^2+2\ell(t)})\right)\right).$$ 
 Second, conditionally on $(Z_u, u\le t)$, the probability that $Z$ jumps from $Z_t=i$ to $j$ in the time interval 
 $[t,t+dt]$ is 
 $$W_{ij}\sqrt{\frac{\phi_j^2+2\ell_j(t)}{\phi_{i}^2+2\ell_{i}(t)}}\,dt
 =W_{ij}\frac{G_j(t)}{G_i(t)}\,dt.$$
 
 Therefore the probability that, at time $t$, $Z$ has followed a path $Z_0=x_0$, $x_1$, $\ldots$, $Z_t=x_n$ with jump times
respectively in $[t_i,t_i+dt_i]$, $i=1\ldots n$,  where $t_0=0<t_1<\ldots<t_n<t=t_{n+1}$, is 
\begin{align*}
&\exp\left(F(\phi)-F(\sqrt{\phi^2+2\ell(t)})\right)
\prod_{i=1}^{n}W_{x_{i-1}x_i}\frac{G_{x_i}(t_i)}{G_{x_{i-1}}(t_i)}\,dt_i\\
&=\exp\left(F(\phi)-F(\sqrt{\phi^2+2\ell(t)})\right)\frac{G_{X_t}(t)}{G_{x_0}(0)}
\prod_{i=1}^{n}W_{x_{i-1}x_i}\,dt_i,
\end{align*}
where we use that $G_{x_i}(t_i)=G_{x_i}(t_{i+1})$, since $Z$ stays at site $x_i$ on the time interval $[t_i,t_{i+1}]$. 

On the other hand, the probability that, at time $t$, $X$ has followed the same   path with jump times in the same intervals is
\begin{align*}
\exp\left(-\sum_{i,j: j\sim i}W_{ij}\ell_i\right)\prod_{i=1}^{n}W_{x_{i-1}x_i}\,dt_i,
\end{align*}
which concludes the proof.
\end{proof}

Note that Theorem \ref{rnd}  can be used to show exchangeability of the VRJP, and provides a martingale for the Markov Jump Process, similar to $M_t^\Phi$ in \eqref{defmart}. 

 Recall that it is shown in \cite{sabot-tarres} that the time-changed VRJP $(Z_t)_{t\ge0}$ is a mixture of Markov Jump Processes, i.e. that there exist random variables $(U_i)_{i\in V}\in \Hh_0:=\{\sum_{\in V} u_i=0\}$ with a sigma supersymmetric hyperbolic distribution with parameters $(W_{ij}\phi_i\phi_j)_{\{i,j\}\in E}$ (see Section 6 of \cite{sabot-tarres} and \cite{dsz})
such that, conditionally on $(U_i)_{i\in V}$, $Z_t$  is a Markov jump process starting from  $z$, with jump rate from  $i$ to $j$
$$
W_{i,j}e^{U_j-U_i}.
$$
In particular, the discrete time process corresponding to the VRJP observed at jump times is exchangeable, and is a mixture of reversible Markov chains with conductances
$W_{i,j}e^{U_i+U_j}$.
\subsection{The reversed VRJP and its Radon-Nykodim derivative}
\label{sec:drvrjp}
\begin{definition}
\label{def:trrjp}
Given positive conductances on the edges of the graph $(W_e)_{e\in E}$ and initial positive local times $(\Phi_i)_{i\in V}$, the reversed vertex-reinforced jump process (RVRJP) is a continuous-time process $(\tilde{Y}_t)_{0\le t\le S}$, starting at time $0$ at some vertex $i_0\in V$ such that, if $\tilde{Y}$ is at a vertex $i\in V$ at time $t$, then, conditionally on $(\tilde{Y}_s, s\le t)$, the process jumps to a neighbour $j$ of $i$ at rate $W_{i,j}\tilde{L}_j(t)$, where
$$\tilde{L}_j(t):=\Phi_j-\int_0^t {\textnormal{\1}}_{\{\tilde{Y}_s=j\}}\,ds,$$
defined up until the stopping time $\tilde S$ where one of the local times hits $0$, i.e. 
$$\tilde{S}=\inf\{t\in\R :  \tilde{L}_j(t)=0\tx{ for some }j\}.$$
\end{definition}

Similarly as for $Y$, let us define the increasing functional
$$\tilde{D}(s)=\frac{1}{2}\sum_{i\in V} (\Phi_i^2-\tilde{L}_i^2(s)),$$
define the time-changed VRJP
$$\tilde{Z}_t= \tilde{Y}_{ \tilde{D}^{-1}(t)}.$$
and let, for all $i\in V$ and $t\ge0$,  $\ell_i^{\tilde Z}(t)$ be the local time of $\tilde Z$ at time $t$.

Then, similarly as in Lemma \ref{markovtime}, conditionally on the past at time $t$, $\tilde{Z}$ jumps from $\tilde{Z}_t=i$ to a neighbour $j$ at rate
$$
W_{i,j} \sqrt{\frac{\Phi_j^2-2{\ell}^{\tilde Z}_j(t)}{\Phi_i^2-2\ell^{\tilde Z}_i(t)}},
$$
and $\tilde{Z}$ stops at time
$$\tilde{T}=\tilde{D}(\tilde{S})=\inf\left\{t\ge0;\,\ell^{\tilde Z}_i(t)=\frac{1}{2}\Phi_i^2\tx{ for some }i\in V\right\}.$$

Let $\P_{\Phi,x_0,t}^{\tilde Z}$ be the distribution of $(\tilde{Z}_t)_{t\ge0}$ on the time interval $[0,t\wedge \tilde T]$, starting from $x_0$ and initial 
condition $\Phi$.

An easy adaptation of the proof of Theorem \ref{rnd} shows 
\begin{theorem}
\label{trrnd}
The law of the time-reversed VRJP $\tilde{Z}$ on the interval $[0,t\wedge \tilde{T}]$  is absolutely continuous with respect to the law of the MJP $X$ with rates $W_{i,j}$, with Radon-Nikodym derivative given by
$$
{d \P_{\Phi,x_0,t}^{\tilde Z}\over d\P_{x_0,t}}=
{e^{-\demi \left(\Ec(\sqrt{\Phi^2-2\ell(t\wedge T)}, \sqrt{\Phi^2-2\ell(t\wedge T)})
-\Ec(\Phi,\Phi)\right)}}{\prod_{j\neq x_0}\Phi_j\over \prod_{j\neq X_{t\wedge T}}\sqrt{\Phi_j^2-2\ell_j(t\wedge T)}},
$$
where $\ell_j(t)$ (resp. $T$) is the local time of $X$ at time $t$ and site $j$ (resp. the stopping time) defined in \eqref{deflt} (resp. in \eqref{defT}).
\end{theorem}
Hence, the Radon-Nikodym derivative of the time-reversed VRJP with respect to the MJP is the martingale that appears in the proof of Theorem \ref{rk2}, more precisely 
$$
{d \P_{\Phi,x_0,t}^{\tilde Z}\over d\P_{x_0,t}}=
\frac{M^\Phi_{t\wedge T}}{M^\Phi_0},
$$
with the notations of Section 2.
Note that this Radon-Nikodym derivative involves the martingale $M$ with positive spins $\sigma\equiv +1$. 
The "magnetized" inverse VRJP, defined in Section 1 and related to the inversion of Ray-Knight in Theorem \ref{inverse1}, involves
the sum all the martingales $M^\sphi$: this is the purpose of next section. 

\section{Proof of lemma \ref{fin} and Theorem \ref{inverse1}}\label{proof-inverse}

% As it appears in Section 3, the process $\tY$ is defined up to exploding time $T$. The process $\tY$ may ends at a point different from
% the starting point $x_0$. The extra magnetization term which appears in $\check Y$ implies that $\check Y$ ends at $x_0$, cf Lemma \ref{fin},
% and is related to the cancellation that appears in the proof of Theorem \ref{rk2}.

The proofs of Lemma \ref{fin} and Theorem \ref{inverse1} rely on a time change of the process $\check Y$ which is in fact the same time change as the one
appearing in Section 3 for $\tY$ : let us define
$$\check{D}(s)=\frac{1}{2}\sum_{i\in V} (\Phi_i^2-{\check L}_i^2(s)),$$
define the time-changed VRJP
$${\check Z}_t= {\check Y}_{ {\check D}^{-1}(t)}.$$
and let, for all $i\in V$ and $t\ge0$,  ${\ell}^{\check Z}_i(t)$ be the local time of $\check Z$ at time $t$.

Then, similarly to Lemma \ref{markovtime}, conditionally on the past at time $t$, $\check{Z}$ jumps from $\check{Z}_t=i$ to a neighbour $j$ at rate
$$
W_{i,j} \sqrt{\frac{\Phi_j^2-2{\ell}^{\check Z}_j(t)}{\Phi_i^2-2{\ell}^{\check Z}_i(t)}}
{<\sigma_j>_{(t)}\over <\sigma_i>_{(t)}},
$$
where we write $<\cdot >_{(t)}$ for $<\cdot >_{D^{-1}(t)}$ according to the notation of Section \ref{intro} : more precisely,  $<\cdot >_{(t)}$ is the expectation for the Ising model with interaction
$$
J_{i,j}(D^{-1}(t))= W_{i,j}\sqrt{{\Phi_i^2-2{\ell}^{\check Z}_i(t)}}\sqrt{{\Phi_j^2-2{\ell}^{\check Z}_j(t)}}
$$
since the vectors of local times ${\ell}^{\check Z}$ and $\check L$ are related by the formula 
\begin{eqnarray}\label{changelt}
{\ell}^{\check Z}(t)= \demi (\Phi^2-{\check L}(D^{-1}(t))).
\end{eqnarray}
Clearly, this process is well defined up to time
$$\check{T}=\check{D}(\check{S})=\inf\left\{t\ge0;\,{\ell}^{\check Z}_{i}(t)=\frac{1}{2}\Phi_i^2\tx{ for some }i\in V\right\}.$$
Lemma \ref{fin} tells that $\check Z_{\check T}=x_0$. 

We denote by $\P^{\check Z}_{\Phi, z}$ the law of the process $\check Z$ starting from the initial condition $\Phi$ and initial state $z$
up to the time $\check T$ (as for $\check Y$ this law depends on the choice of $x_0$). 

We now prove a more precise version of Theorem \ref{inverse1}, giving a description of the conditional law of the full process.
\begin{theorem}\label{inverse2}
With the notations of Theorem \ref{inverse1}, under $\P_{x_0}\left( \cdot | \Phi\right)$, $(\left(X_t)_{t\in [0, \tau_u]}, \phi\right)$ has the law of 
$(\left(({\check Z}(t))_{t\in [0,T]}, \sigma
%\sqrt{\Phi^2-2l(T)})
\Phi(T)\right)$ where $\check Z$ is distributed under $\P_{\Phi, x_0}^{\check Z}$ and 
$\sigma$ is distributed according to the Ising model with interation $W_{i,j}\Phi_i(T)\Phi_j(T)$.
\end{theorem}

We will adopt the following notation
\begin{eqnarray}\label{Phit}
\Phi_i(t)=\sqrt{\Phi_i^2-2l_i^{\check Z}(t)}={\check L}_i(D^{-1}(t)).
\end{eqnarray}

Recall that $M_t^{\phi}$, $N_t^\Phi$ and $T$ are  the processes (starting with the initial conditions $\phi$ and $\Phi$) and stopping times defined respectively in \eqref{defmart}, \eqref{defN} and \eqref{defT}, as a function of the path of the Markov process $X$ up to time $t$. 
The proof of Theorem \ref{inverse2} is based on the following lemma. 
% $N_{t\wedge T}^\Phi$ is, up to a multiplicative constant, the Radon-Nykodim derivative of $\P_{\Phi,x_0}^{\check Z}$ with respect to $\P_{x_0}$.

\begin{lemma}\label{N}
We have:
\begin{enumerate}
\item[{\bf(i)}] For all $t\le T$,
$$
N^{\Phi}_t =e^{\sum_{i\in V}W_i(\ell_i(t)-\demi \Phi_i^2)} F(D^{-1}(t))<\sigma_{X_t}>_{(t)}\left(\frac{\prod_{j\ne x_0} \Phi_j(0)}{\prod_{j\ne X_{t}}\Phi_j(t)}\right),
$$
where $F(D^{-1}(t))$ (resp. $<\cdot >_{(t)}$) corresponds to the partition function (resp. distribution) of the Ising model with interaction $J_{i,j}(D^{-1}(t))=J_{i,j} \Phi_i(t)\Phi_j(t)$, and $W_i=\sum_{j\sim i} W_{i,j}$.

\item[{\bf(ii)}] 
$
N_{T}=0 
$ 
if $X_{T}\neq x_0$.

\item[{\bf(iii)}]  Under $\P_{Z}$ (the law of the MJP $(X_t)$), $N^\Phi_{t\wedge T}$ is a positive martingale, more precisely, ${N^\Phi_{t\wedge T}/N_0^\Phi}$ 
is the Radon-Nykodim derivative of the measure $\P_{\Phi,z}^{\check Z}$ with respect to the law of the MJP $X$ starting from $z$ and stopped at time $T$.
\end{enumerate}
\end{lemma}
\begin{proof}[Proof of Lemma \ref{N}]
{\bf(i)} We expand the squares in the energy term, which yields
$$
\demi\Ec(\sigma\Phi(t),\sigma \Phi(t))=
-\sum_{\{i,j\}\in E} W_{i,j} \Phi_i(t)\Phi_j(t)\sigma_i\sigma_j+ \demi \sum_{i\in V} W_i(\Phi_i^2-2{\ell}_i(t)),
$$
and the statement follows easily.

{\bf(ii)} Same argument as in Lemma \ref{sign-flip}.
This can also be seen from the expression in {\bf(i)} since in this case all the interactions between $x$ and its neighbors vanish, indeed,
$J_{x,y}(T)= 0$, using $\Phi_x(T)=0$. This implies that the pinning $\sigma_{x_0}=+1$ has no effect on the spin $\sigma_x$, and therefore
by symmetry that  $<\sigma_x>_{(T)}=0$ if $X_T=x\neq x_0$.

{\bf(iii)} The fact that $N^\Phi_t$ is a martingale follows directly from the martingale property of the $M^{\sigma \Phi}_t$, cf Lemma \ref{l1}. It is also a consequence
of the Radon-Nykodim property proved below.
The fact that $N^\Phi_t$ is positive follows from the positive correlation in the Ising model : $<\sigma_x>_{(t)}=<\sigma_{x_0}\sigma_x>_{(t)}\ge 0$,
see for instance \cite{werner}.

The beginning of the proof follows the same line of ideas as in the proof of Theorem \ref{rnd}. 
Similarly, we set
$$
\check G_i(t)= \prod_{j\neq i} {1\over\Phi_j(t)}= \prod_{j\neq i} {1\over \sqrt{{\Phi_j^2-2{\ell}^{\check Z}_j(t)}}},
$$
so that
$$
{\Phi_j(t)\over \Phi_i(t)}={\check G_i(t)\over \check G_j(t)}.
$$

First note that the probability, for the time-changed process $\check Z$, of holding at a site $v\in V$ on a time interval $[t_1,t_2]$ is 
 $$\exp\left(-\int_{t_1}^{t_2}\sum_{j\sim \check Z_u}W_{\check Z_u,j}\frac{\Phi_j(u)<\sigma_j>_{(u)}}{\Phi_{\check Z_u}(u)<\sigma_{\check Z_u}>_{(u)}} du\right).
 $$ 
 Second, conditionally on $(\check Z_u, u\le t)$, the probability that $\check Z$ jumps from $\check Z_t=i$ to $j$ in the time interval 
 $[t,t+dt]$ is 
 $$ W_{ij}\frac{\Phi_j(t)<\sigma_j>_{(t)}}{\Phi_i(t)<\sigma_i>_{(t)}} \,dt.$$
 
 Therefore the probability that, at time $t$, $\check Z$ has followed a path $\check Z_0=x_0$, $x_1$, $\ldots$, $\check Z_t=x_n$ with jump times
respectively in $[t_i,t_i+dt_i]$, $i=1\ldots n$,  where $t_0=0<t_1<\ldots<t_n<t=t_{n+1}$, with $t\le \check T$, is 
\begin{align*}
&\exp\left( -\int_{0}^{t}\sum_{j\sim \check Z_u}W_{\check Z_u,j}\frac{\Phi_j(u)<\sigma_j>_{(u)}}{\Phi_{\check Z_u}(u)<\sigma_{\check Z_u}>_{(u)}} du    \right)
\prod_{i=1}^{n}W_{x_{i-1}x_i}\frac{\Phi_{x_i}(t_i) <\sigma_{x_i}>_{(t_i)}  }{\Phi_{x_{i-1}}(t_i)<\sigma_{x_{i-1}}>_{(t_i)}}\,dt_i\\
=
&\exp\left( -\int_{0}^{t}\sum_{j\sim \check Z_u}W_{\check Z_u,j}\frac{\Phi_j(u)<\sigma_j>_{(u)}}{\Phi_{\check Z_u}(u)<\sigma_{\check Z_u}>_{(u)}} du    \right)
\prod_{i=1}^{n}W_{x_{i-1}x_i}\frac{\check G_{x_{i-1}}(t_i)}{\check G_{x_i}(t_i)}\frac{ <\sigma_{x_i}>_{(t_i)}  }{<\sigma_{x_{i-1}}>_{(t_i)}}\,dt_i\\
=
&\exp\left( -\int_{0}^{t}\sum_{j\sim \check Z_u}W_{\check Z_u,j}\frac{\Phi_j(u)<\sigma_j>_{(u)}}{\Phi_{\check Z_u}(u)<\sigma_{\check Z_u}>_{(u)}} du    \right)
\frac{\check G_{x_0}(0)}{\check G_{\check Z_t}(t)}\prod_{i=1}^{n}W_{x_{i-1}x_i}\frac{<\sigma_{x_i}>_{(t_i)}  }{<\sigma_{x_{i-1}}>_{(t_i)}}\,dt_i\\
%=
%&\exp\left( -\int_{0}^{t}\sum_{j\sim \check Z_u}W_{\check Z_u,j}\frac{G_j(u)<\sigma_j>_{(u)}}{G_{\check Z_u}(u)<\sigma_{\check Z_u}>_{(u)}} du    \right)
%\frac{\check G_{\check Z_t}(t)<\sigma_{\check Z_t}>_{(t)}}{\check G_{x_0}(0)}
%\left(\prod_{i=1}^{n+1}\frac{<\sigma_{x_{i-1}}>_{(t_{i-1})}  }{<\sigma_{x_{i-1}}>_{(t_i)}}\right)\left( \prod_{i=1}^{n}W_{x_{i-1}x_i}\,dt_i\right)\\
%=&\exp\left( -\int_{0}^{t}( \sum_{j\sim \check Z_u}W_{\check Z_u,j}\frac{G_j(u)<\sigma_j>_{(u)}}{G_{\check Z_u}(u)<\sigma_{\check Z_u}>_{(u)}})  + 
%\frac{\frac{\partial}{\partial u} <\sigma_{\check Z_u}>_{(u)}}{<\sigma_{\check Z_u}>_{(u)}} du\right)
%\frac{\check G_{\check Z_t}(t)<\sigma_{\check Z_t}>_{(t)}}{G_{x_0}(0)}\prod_{i=1}^{n}W_{x_{i-1}x_i}\,dt_i\\
\end{align*}
where we use that $\check G_{x_{i-1}}(t_{i-1})=\check G_{x_{i-1}}(t_{i})$, since $Z$ stays at site $x_{i-1}$ on the time interval $[t_{i-1},t_{i}]$. 
We now use that
\begin{align*}
\prod_{i=1}^{n}\frac{<\sigma_{x_i}>_{(t_i)}  }{<\sigma_{x_{i-1}}>_{(t_i)}}
=& \frac{<\sigma_{\check Z_t}>_{(t)}}{<\sigma_{x_0}>_{(0)}}
\prod_{i=1}^{n+1}\frac{<\sigma_{x_{i-1}}>_{(t_{i-1})}  }{<\sigma_{x_{i-1}}>_{(t_i)}}
\\
=&
<\sigma_{\check Z_t}>_{(t)} \exp\left(-\int_0^t  \frac{\frac{\partial}{\partial u} <\sigma_{\check Z_u}>_{(u)}}{<\sigma_{\check Z_u}>_{(u)}} du\right).
\end{align*}
Finally, set
$$
H(t)=F(D^{-1}(t))= \sum_{\sigma\in \{-1,+1\}^{V\setminus\{x_0\}}} \exp\left\{\sum_{\{i,j\}\in E} W_{i,j}\Phi_i(t)\Phi_j(t)\sigma_i\sigma_j\right\}
$$
and,
$$
K(t)=
 \sum_{\sigma\in \{-1,+1\}^{V\setminus\{x_0\}}} \exp\left\{\sum_{\{i,j\}\in E} W_{i,j} \Phi_i(t)\Phi_j(t)\sigma_i\sigma_j \right\} \sigma_{\check Z_t}.
 $$
We have $<\sigma_{\check Z_t}>_{(t)}=K(t)/H(t)$, so that
$$
 \frac{\frac{\partial}{\partial u} <\sigma_{\check Z_u}>_{(u)}}{<\sigma_{\check Z_u}>_{(u)}}
 =
 \frac{\frac{\partial}{\partial u} K(u)}{K(u)}- \frac{\frac{\partial}{\partial u} H(u)}{H(u)}.
 $$
Now, since
$$
\frac{\partial}{\partial u} \left\{\sum_{\{i,j\}\in E} W_{i,j} \Phi_i(u)\Phi_j(u)\sigma_i\sigma_j\right\}=
-\sum_{j\sim \check Z_u} W_{\check Z_u,j} \frac{\Phi_j(u)}{\Phi_{\check Z_u}(u)}\sigma_{\check Z_u}\sigma_j,
$$
we have that
$$
\frac{\frac{\partial}{\partial u} K(u)}{K(u)}=  -\sum_{j\sim \check Z_u} {\Phi_j(u)\over \Phi_{\check Z_u}(u)} 
{<\sigma_{j}>_{(u)}\over  <\sigma_{\check Z_u}>_{(u)}}
$$
 These identities imply that the probability that, at time $t$, 
$\check Z$ has followed a path $\check Z_0=x_0$, $x_1$, $\ldots$, $\check Z_t=x_n$ with jump times
respectively in $[t_i,t_i+dt_i]$, $i=1\ldots n$,  where $t_0=0<t_1<\ldots<t_n<t=t_{n+1}$, with $t\le \check T$, is 
$$
{H(t)\over H(0)} <\sigma_{\check Z_t}>_{(t)}
\frac{\check G_{x_0}(0)}{\check G_{\check Z_t}(t)}\prod_{i=1}^{n}W_{x_{i-1}x_i}\,dt_i
=
{N^\Phi(t)\over N^\Phi(0)} \exp\left\{-\sum_{i\in V} W_i \ell_i(t)\right\}\prod_{i=1}^{n}W_{x_{i-1}x_i}\,dt_i
$$
where in the last equality we used Lemma \ref{N}, {\bf(i)}. Finally, 
the probability that, at time $t$, the Markov jump process $X$ has followed the same path with jump times in the same intervals is
\begin{align*}
\exp\left(-\sum_{i\in V}W_{i}\ell_i\right)\prod_{i=1}^{n}W_{x_{i-1}x_i}\,dt_i.
\end{align*}
This exactly tells that the Radon-Nykodim derivative of $Z_{t\wedge \check T}$ under $ \P^{\check Z}_{\Phi}$ with respect to the law 
$\P$ of the Markov jump process is ${N^\Phi(t\wedge \check T)\over N^\Phi(0)}$.
\end{proof}
\begin{proof}[Proof of Lemma \ref{fin}]
By {\bf(ii)} and {\bf(iii)} of Lemma \ref{N}, we have, by the optional stopping theorem, 
\begin{eqnarray*}
\E^{\check Z}_{\Phi, x_0}\left( \indic_{\{{\check Y}_{\check S}\neq x_0\}}\right)&=&
\E^{\check Z}_{\Phi, x_0}\left( \indic_{\{{\check Z}_{\check T}\neq x_0\}}\right)
\\
&=&
\E_{x_0}\left( N_T^\Phi \indic_{\{X_T\neq x_0\}}/N_0^\Phi\right)=\E_{x_0}\left( N_T^\Phi/N_0^\Phi\right)
\\
&=&
0
\end{eqnarray*}
\end{proof}
\begin{proof}[Proof of Theorem \ref{inverse2}]
%We are now in a position to prove theorem \ref{inverse2}.
Let $\psi((X_t)_{t\in [0,\tau_u]}, \phi)\stackrel{notation}{=}\psi(X, \phi)$ and $G(\Phi)$ be test functions. We are interested in the following expectation
\begin{eqnarray}\label{e1}
\E_{x_0}\otimes P^{G,U}\left( \psi(X,\phi) G(\Phi) \right)
=\E_{x_0}\left( \int_{\R^{V\setminus\{x_0\}}} \psi(X,\phi) G(\Phi) Ce^{-\demi \Ec(\phi, \phi)} d\phi\right),
\end{eqnarray}
where, as in the proof of Theorem \ref{rk2}, $C$ is the normalizing constant of the Gaussian free field.
Recall that $\Phi= \sqrt{ \phi^2+2l(\tau_u)}$  and set $\sigma = \sign (\phi)$. 
As in the proof of Theorem \ref{rk2}  we change to variables $\Phi$. Following the computation at the beginning of the proof of Theorem
\ref{rk2} up to equation (\ref{eq4}), we deduce  that (\ref{e1}) is equal to
%\begin{eqnarray}\label{e2}
%C \int_{\R_+^{V\setminus\{x_0\}}} F(\Phi) \E_{x_0}\left( \sum_\sigma \psi(X,\phi) M_T^{\sphi} \indic_{X_T=x_0}\right) d\Phi,
%\end{eqnarray}
%with $\phi=\sigma\sqrt{\Phi^2-2l(T)}$.
%Summing on the possible signs of $\tilde \phi$, we can integrate on $\Phi$ instead and we get that (\ref{e2}) is equal to
\begin{eqnarray}\label{e3}
C\int_{\R_+^{V\setminus\{x_0\}}} G(\Phi) \E_{x_0}\left( \sum_{\sigma} \psi(X,\sigma\Phi(T)) M_T^{\sigma\Phi} \indic_{\{X_T=x_0\}}\right) d\Phi
\end{eqnarray}
If $X_T=x_0$ then, using that $\s_{X_t}=\s_{x_0}=1$ and the expansion in the proof of Lemma \ref{N} {\bf(i)}, we deduce that
$$M_T^{\sigma\Phi}=N_T^{\Phi}\frac{e^{\sum_{\{i,j\}\in E} W_{i,j} \Phi_i(T)\Phi_j(T) \sigma_i\sigma_j}}{F(D^{-1}(T))}$$
and, therefore,
\begin{eqnarray*}
\sum_{\sigma} \psi(X,\sigma\Phi(T)) M_T^{\sigma\Phi} 
&=& N_T^{\Phi}{1\over F(D^{-1}(T))} \sum_{\sigma} \psi(X,\sigma\Phi(T)) e^{\sum_{\{i,j\}\in E} W_{i,j} \Phi_i(T)\Phi_j(T) \sigma_i\sigma_j}
\\
&=& N_T^{\Phi} <\psi(X,\sigma\Phi(T))>_{(T)}.
\end{eqnarray*}
This implies that
\begin{eqnarray*}
(\ref{e3})
&=&
C\int_{\R_+^{V\setminus\{x_0\}}} G(\Phi) \E_{x_0}\left( <\psi(X,\sigma\Phi(T))>_{(T)}N_T^\Phi \indic_{\{X_T=x_0\}}\right) d\Phi
\\
&=&
C\int_{\R_+^{V\setminus\{x_0\}}} G(\Phi) N_0^\Phi \check\E^{VRJP}_{x_0}\left( <\psi(\check Z,\sigma\Phi(T))>_{(T)} \right) d\Phi
\end{eqnarray*}
where in the last equality we used Lemma \ref{N} {\bf(ii)}-{\bf(iii)}.
Since 
$$N_0^\Phi =\sum_{\sigma\in \{-1,+1\}^{V\setminus\{x_0\}}} \exp\left\{-\frac{1}{2}\Ec(\sigma\Phi,\sigma \Phi)\right\},
$$ 
it implies that
$CN_0^\Phi$ is the density of $\Phi$ since by Theorem \ref{rk2} we have $\Phi\stackrel{law}{=} |\sqrt{2u}+\phi|^2$ where $\phi$ has the law of the Gaussian free field $P^{G,U}$.
This exactly means that
$$
\E_{x_0}\otimes P^{G,U}\left( \psi(X,\phi) | \Phi \right)=  \E^{\check Z}_{\Phi,x_0}\left( <\psi(\check Z,\sigma\Phi(T))>_{(T)}\right).
$$
\end{proof}
\begin{proof}[Proof of Theorem \ref{inverse1}]
From Theorem \ref{inverse2}, we know that conditionally on $\Phi$, $(\ell,\phi)$ has the law of
$(\ell(T), \sqrt{\Phi^2-2\ell(T)})$, where $\ell(T)$ is the local time of $\check Z$ under $\P^{\check Z}_{\Phi,x_0}$.
If we change back to the process $\check Y$ we have, using \eqref{changelt},
$$
\check L(S)=\sqrt{\Phi^2-2\ell(T)},
$$
hence, $\lll((\ell,\phi)| \Phi)$ is the law of $(\demi(\Phi^2-L^2(S)), L(S))$ for initial conditions $(\Phi, x_0)$.

\end{proof}

\section{Inversion of the generalized first Ray-Knight theorem}

We use the same notation as in the first section. The generalized first Ray-Knight theorem concerns the local time
of the Markov jump process starting at a point $z_0\neq x_0$, stopped at its first hitting time of $x_0$.
Denote by
$$H_{x_0}=\inf\{t\ge 0, \;\; X_t=x_0\},
$$
the first hitting time of $x_0$.
\begin{theorem}\label{rk1}
For any $z_0\in V$ and any $s> 0$,
\begin{align*}
&\left(\ell_x(H_{x_0})+\frac{1}{2}(\phi_x+s)^2\right)_{x\in V} \tx{ under }\P_{z_0}\otimes P^{G,U}, \tx{ has the same "law" as}\\
&\left(\frac{1}{2}(\phi_x+s)^2\right)_{x\in V}\tx{ under }(1+{\phi_{z_0}\over s})P^{G,U}.
\end{align*}
\end{theorem}
\begin{remark} This theorem is in general stated for $s\neq 0$, but obviously we do not loose generality by restricting to
$s>0$.
\end{remark}
This formally means that for any test function $g$, 
\begin{equation}
\label{test1}
\int 
g \left((\ell_x(H_{x_0})+\frac{1}{2}(\phi_x+s)^2)_{x\in V} \right) d\P_{z_0}\otimes P^{G,U}
=
\int 
g \left((\frac{1}{2}(\phi_x+s)^2)_{x\in V} \right) (1+{\phi_{z_0}\over s})dP^{G,U}.
\end{equation}
Remark that the measure $(1+{\phi_{z_0}\over s})P^{G,U}$ has mass 1 (since $\phi_{z_0}$ is centered) but is not positive.
In fact, since the integrand depends only on $\vert \phi_x+s\vert$, $x\in V$, everything can be written in terms of a positive measure.
Indeed, if $\sigma_x=\sign(\phi_x+s)$, then conditionally on $\vert \phi_x +s\vert$, $x\in V$, $\sigma$ has the law of an Ising model with interaction
$J_{i,j} = W_{i,j} \vert \phi_i +s\vert \vert \phi_j +s\vert$ and boundary condition $\sigma_{x_0}= +1 $. This implies
that the right hand side of \eqref{test1} can be written equivalently as 
$$
\int  g \left( (\frac{1}{2}(\phi_x+s)^2)_{x\in V} \right) {<\sigma_{z_0}> \over s}\vert s+ \phi_{z_0}\vert d P^{G,U}$$
where $<\sigma_{z_0}>$ denotes the expectation of $\sigma_{z_0}$ with respect to the Ising model described above.
Since $\sigma_{x_0}=+1$, we have that ${<\sigma_{z_0}> \over s}\ge 0$, and ${<\sigma_{z_0}> \over s}\vert s+ \phi_{z_0}\vert d P^{G,U}$
is a probability measure.

We give now a counterpart of Theorem \ref{inverse1} for the generalized first Ray-Knight theorem. Consider the 
process $\check Y$ defined in Section 1, starting from a point $z_0$.
Denote by $\check H_{x_0}$ the first hitting time of $x_0$ by the process $\check Y$. 

Obviously, Lemma \ref{fin} implies
the following Lemma \ref{fin1}.
\begin{lemma}\label{fin1}
Almost surely $\check H_{x_0}\le S$, where $S$ is defined in (\ref{S}).
\end{lemma}
\begin{theorem}\label{inverse-rk1}
With the notation of Theorem \ref{rk1}, let
$$
\Phi_z=\sqrt{2\ell_z(H_{x_0})+(\phi_z+s)^2}.
$$
Under $\P_{z_0}\otimes P^{G,U}$, we have
$$
\lll\left( \phi+s | \Phi\right)\stackrel{law}{=}  (\s\check L(\check H_{x_0})),
$$
where $\check L(\check H_{x_0})$ is distributed under $\P^{\check Z}_{\Phi, z_0}$
%the law of $\check Y$ starting at initial vertex $z_0$ and 
%with initial condition $\Phi$, 
and,  conditionally on $\check L(\check H_{x_0})$, 
 $\sigma$ is distributed according to the distribution of the Ising model with interaction $J_{i,j}(\check H_{x_0})=W_{i,j}\check L_i(\check H_{x_0})\check L_j(\check H_{x_0})$
 and boundary condition $\sigma_{x_0}=+1$.
\end{theorem}

Similarly as for the generalized second Ray-Knight theorem, Theorem \ref{inverse-rk1} is a consequence of the following more precise result. Let us consider, as in Section 4, the time-changed version $\check Z$ of the process $\check Y$. 
\begin{theorem}\label{inverse2-rk1}
With the notation of Theorem \ref{inverse-rk1}, under $\P_{z_0}\left( \cdot | \Phi\right)$, $(\left(X_t)_{t\in [0, H_{x_0}]}, \phi+s\right)$ has the law of 
$(\left(({\check Z}(t))_{t\in [0,\check H_{x_0}]}, \sigma
%\sqrt{\Phi^2-2l(T)})
\Phi(\check H_{x_0})\right)$ where $\check Z$ is distributed under $\P_{\Phi, z_0}^{\check Z}$,  $\check H_{x_0}$ is the first hitting
time of $x_0$ by $\check Z$, and $\sigma$ is distributed according to the Ising model with interation $W_{i,j}\Phi_i(\check H_{x_0})\Phi_j(\check H_{x_0})$
and boundary condition $\sigma_{x_0}=+1$.
\end{theorem}
\begin{proof} We only sketch the proof since it is very similar to the proof of  Theorem \ref{inverse2}.
Let $\psi((X_t)_{t\in [0, H_{x_0}]}, \phi+s)\stackrel{notation}{=}\psi(X, \phi+s)$ and $G(\Phi)$ be positive test functions. We are interested in the following expectation
\begin{eqnarray}\label{ee1}
\;\;\;\;\; \E_{z_0}\otimes P^{G,U}\left( \psi(X,\phi+s) G(\Phi) \right)
=\E_{z_0}\left( \int_{\R^{V\setminus\{x_0\}}} \psi(X,\phi+s) G(\Phi) Ce^{-\demi \Ec(\phi, \phi)} d\phi\right)
\end{eqnarray}
where, as in the proof of Theorem \ref{rk2}, $C$ is the normalizing constant of the Gaussian free field.
Recall that $\Phi= \sqrt{ (\phi+s)^2+2\ell(H_{x_0})}$, set $\sigma = \sign (\phi+s)$ and define
$$
T'=S\wedge H_{x_0}.
$$
As in the proof of Theorem \ref{rk2}  we change to variables $\Phi$. An easy adaptation of the computation in the proof of Theorem \ref{rk2} 
up to equation (\ref{eq4}) yields  that (\ref{ee1}) is equal to
\begin{eqnarray}\label{e3}
C\int_{\R_+^{V\setminus\{x_0\}}} G(\Phi) \E_{z_0}\left( \sum_{\sigma} \psi(X,\sigma\Phi(T'))
% \left({\sigma_{z_0}\Phi_{z_0}\over s}\right) 
M_{T'}^{\sigma\Phi} \indic_{\{X_{T'}=x_0\}}\right) d\Phi,
\end{eqnarray}
%the extra term $\left({\sigma_{z_0}\phi_{z_0}\over s}\right)$ coming from the fact that the process start at a point $z_0$ not equal to $x_0$.
As in the proof of Theorem \ref{inverse2} we have that, if $X_{T'}=x_0$, then
\begin{eqnarray*}
\sum_{\sigma} 
%\sigma_{z_0} 
\psi(X,\sigma\Phi(T')) M_{T'}^{\sigma\Phi} 
&=& N_{T'}^{\Phi}{1\over F(D^{-1}(t))} \sum_{\sigma} 
%\sigma_{z_0} 
\psi(X,\sigma\Phi(T')) e^{\sum_{\{i,j\}\in E} W_{i,j} \Phi_i(T')\Phi_j(T') \sigma_i\sigma_j}
\\
&=& N_{T'}^{\Phi} <
%\sigma_{z_0} 
\psi(X,\sigma\Phi(T'))>_{(T')}.
\end{eqnarray*}
This implies that
\begin{eqnarray*}
(\ref{e3})
&=&
C\int_{\R_+^{V\setminus\{x_0\}}} G(\Phi) 
%\left({\Phi_{z_0}\over s}\right) 
\E_{z_0}\left( <
%\sigma_{z_0}
\psi(X,\sigma\Phi(T'))>_{(T')}N_{T'}^\Phi \indic_{X_{T'}=x_0}\right) d\Phi
\\
&=&
C\int_{\R_+^{V\setminus\{x_0\}}} G(\Phi) 
%\left({\Phi_{z_0}\over s}\right) 
N_0^\Phi \E^{\check Z}_{\Phi,z_0}\left( <
%\sigma_{z_0}
\psi(\check Z,\sigma\Phi(T'))>_{(T')} \right) d\Phi,
\end{eqnarray*}
using in the last equality  an easy adaptation of Lemma \ref{N} {\bf(ii)-(iii)} for time $T'$.
Now
$$N_0^\Phi =\sum_{\sigma\in \{-1,+1\}^{V\setminus\{z_0\}}}\left({\sigma_{z_0}\Phi_{z_0}\over s}\right)  \exp\left\{-\frac{1}{2}\Ec(\sigma\Phi,\sigma \Phi)\right\},
$$ 
which implies that
$CN_0^\Phi$ is the density of $\Phi$ since by Theorem \ref{rk1} we have $\Phi\stackrel{law}{=} |\phi+s|^2$ under 
$(1+{\phi_{z_0}\over s}) P^{G,U}$.
This exactly means that
$$
\E_{z_0}\otimes P^{G,U}\left( \psi(X,\phi) | \Phi \right)=  \E^{\check Z}_{\Phi,z_0}\left( <\psi(\check Z,\sigma\Phi(T'))>_{(T')}\right).
$$

\end{proof}

\medskip
\noindent
{\bf Acknowledgment.} We are grateful to Alain-Sol Sznitman and Jay Rosen for several useful comments on a first version of the manuscript. We thank also Yuval Peres
for interesting discussions.

\footnotesize
\nocite{*}
\bibliographystyle{plain}
\bibliography{ray-knight}

\end{document}